\documentclass[12pt]{article}

\usepackage{tikz,amsthm,amsmath,amstext,amssymb,epsfig,euscript, mathrsfs, dsfont,pspicture,multicol,graphpap,graphics,graphicx,times,enumerate,subfig,sidecap,
wrapfig, color}
\usepackage{enumerate}

\def\ms{\medskip}

\def\1{{\mathds 1}}
\usepackage{hyperref}
\hypersetup{backref,  pdfpagemode=FullScreen} 
\usepackage{esint}

\def\barintgerm_#1{\mathchoice
{\mathop{\vrule width 6pt height 3 pt depth -2.5pt
\kern -8.8pt \intop}\nolimits_{#1}}%
{\mathop{\vrule width 5pt height 3 pt depth -2.6pt
\kern -6.5pt \intop}\nolimits_{#1}}%
{\mathop{\vrule width 5pt height 3 pt depth -2.6pt
\kern -6pt \intop}\nolimits_{#1}}%
{\mathop{\vrule width 5pt height 3 pt depth -2.6pt \kern -6pt
\intop}\nolimits_{#1}}}

\theoremstyle{plain}
\newtheorem{theorem}{Theorem}

\newtheorem{lemma}[theorem]{Lemma}

\theoremstyle{definition}
\newtheorem{example}[theorem]{Example}
\newtheorem{definition}[theorem]{Definition}
\newtheorem{remark}[theorem]{Remark}

\numberwithin{equation}{section}
\numberwithin{theorem}{section}

\begin{document}

\title{K\"ahler geometry of bounded pseudoconvex  Hartogs
domains\thanks{This project was partially  supported by NSF of China (No.
11371257).
}}
\author{Yihong Hao
\footnote{Corresponding author}
\and  An Wang
}

\maketitle

\abstract

Let $\Omega$  be a bounded pseudoconvex Hartogs domain. There exists a natural complete
K\"ahler metric $g^{\Omega}$ in terms of its defining function.
In this paper, we study two problems. The first one  is determining when $g^{\Omega}$ is Einstein or extremal.
The second one is the existence of holomorphic isometric immersions of
$(\Omega, g^{\Omega})$ into finite or infinite dimensional complex  space forms.

\ms

\ms

\tableofcontents

\section{Introduction}\label{sec:1}

In \cite{Cheng1980}, Cheng and Yau proved that there exist complete K\"ahler-Einstein metrics with
negative Ricci curvature on strictly pseudoconvex domains with $C^{k}$ boundary, $k\geq 5$.
This metric is unique if the metric is normalized by multiplication by a constant to have
the eigenvalues of the Ricci tensor identically $-1$. In \cite{Mok1983},
Mok and Yau extended this result to bounded pseudoconvex domains.
The proof in \cite{Cheng1980} involves solution by the continuity method of a Monge-Amp\`ere equation.
The solution actually is a special K\"ahler potential function of the K\"ahler-Einstein metric.
On homogeneous domains, the Bergman metric is the K\"ahler-Einstein metric with negative Ricci curvature $-1$.
Up to a constant, the Bergman kernel is the solution of the Monge-Amp\`ere equation introduced by Cheng and Yau.
In general, it is very difficult to obtain the solution.
In \cite{Wang2006}, Yin, Zhang, Roos and the second author of this paper obtained the explicit solutions for some nonhomogeneous
Cartan-Hartogs domains
by reducing the higher order nonlinear partial differential equation to an ordinary differential equation (see the definition
of Cartan-Hartogs domain in Example \ref{ex:CH}).
 In \cite{Wang2014}, we
generalized this result for some generalized Cartan-Hartogs domains over bounded symmetric domains.
Notice that the K\"ahler-Einstein metric we obtain is just the natural
complete K\"ahler metric given by a defining function of the domain.
This K\"ahler metric was  constructed  by  Cheng and Yau on the strictly pseudoconvex domain $\Omega$ with
$C^{k}$, $k\geq 5$ boundary in $\mathbb{C}^{n}$ firstly. To be precise,
let $F$ be a defining function for $\Omega$.
Define $g^{\Omega}$ whose K\"ahler potential is
$-\log (-F).$
Then they proved that $(\Omega, g^{\Omega})$ is a complete K\"ahler manifold whose Ricci tensor is ``asymtotically Einstein''.
Moveover, one can get a complete K\"ahler-Einstein metric by perturbing the metric $g^{\Omega}$.
On some special domains,  we can describe the distinction between
this natural K\"ahler metric and the Einstein-K\"ahler metric.
In our case, the  boundary condition  is no more necessary.

Recall the following definition firstly.
Let $D\subset \mathbb{C}^{d}$ be a domain and $\varphi$ be a continuous positive function on $D$.
The domain
\begin{equation}\label{equ:Hartogs domain}
\Omega=\Big\{(z_{0}, z)\in \mathbb{C}^{d_{0}}\times D : ||z_{0}||^{2}<\varphi(z)\Big\}
\end{equation}
is called a Hartogs domain over  $D$ with $d_{0}$-dimensional fibers.
Hartogs domains have  been investigated by many mathematicians for studying many problems in several complex variables. It is easy to see that $\Omega$ is pseudoconvex
if and only if $D$ is pseudoconvex and $-\log \varphi$ is plurisubharmonic.
Conversely,
Hartogs domains  can also be used to characterize pseudoconvex domains.
Let $D$ be a domain, then $D$ is pseudoconvex if and only if
the Hartogs-like domain $\{(z_{0},z)\in  \mathbb{C}^{d_{0}}\times D: z+\lambda z_{0}\in D, \lambda \in \mathbb{C}, |\lambda|\leq1  \}$ over $D$
 with balanced fibers is pseudoconvex \cite{Nikolov2013}.

From \cite{Loi2009}, we know $\mathrm{Engli\check{s}}$ had considered
the  bounded simply connected pseudoconvex  Hartogs  domain $(\Omega, g^{\Omega})$,
where $g^{\Omega}$ is the natural K\"ahler metric
whose K\"ahler potential is $-\log (\varphi(z)-||z_{0}||^{2})$.
Let $g_{B}^{\Omega}$  be the Bergman metric on $\Omega$.
He proved that if $g^{\Omega}=\lambda g_{B}^{\Omega}$
for some $\lambda\in \mathbb{R}^{+}$, then $g^{\Omega}$ is K\"ahler-Einstein.
By the results in \cite{Wang2014}, we know that
there also exist many examples whose natural K\"ahler metrics $g^{\Omega}$ are Einstein, while
$g^{\Omega}\neq\lambda g_{B}^{\Omega}$ for any $\lambda \in \mathbb{R}^{+}$.
Notice that many important properties of Hartogs domain can be characterized by the base such as
convex, smoothly (or real-analytically) bounded,
Bergman kernel \cite{Ligocka1989}, the Bergman completeness of non hyperconvex
domains \cite{Jarnicki2000}, $D^{*}$-extension property \cite{Thai1998}, $k$-hyperbolic \cite{Thai2000},
$K$-exhaustive \cite{Jarnicki2013} and so on.
So we prefer to  characterize the K\"ahler-Einstein metric of $(\Omega, g^{\Omega})$ in terms of its base.

Our first result is Theorem \ref{thm:1.1} which characterizes
two canonical metrics of the bounded pseudoconvex Hartogs domain in terms of the base.
The first one is the K\"ahler-Einstein metric. The second one is the extremal K\"ahler metric
which is one of the generalizations of K\"ahler-Einstein metric.
It was introduced by Calabi \cite{Calabi1982,Calabi1985}
for finding the canonical representant of a given K\"ahler classes $[\omega]$ of a complex  compact K\"ahler manifold $(M, J)$.
In the non compact case, the problem of finding extremal metrics is quite natural but difficult \cite{Xu1996}.
In \cite{Chang2002},  Chang  proved the existence of extremal metrics of a complete noncompact smooth surface.
In  \cite{Loi2010},  Loi and Zudda proved that the only extremal metric  is the hyperbolic metric for a strongly pseudoconvex Hartogs domain.
In \cite{Zedda2012}, Zedda considered the  Cartan-Hartogs domain endowed with a natural K\"ahler metric. He  proved that this metric  is
extremal if and only if it is Einstein.  Our theorem extends Zedda's result
for any bounded pseudoconvex Hartogs domain.
\begin{theorem}\label{thm:1.1}
Let
\begin{equation}\label{equ:Omega}
\Omega=\Big\{(z_{0}, z)\in \mathbb{C}^{d_{0}}\times D : ||z_{0}||^{2}<\varphi(z)\Big\}
\end{equation}
be a  bounded pseudoconvex Hartogs domain over $D\subset\mathbb{C}^{d}$,
where $-\log \varphi$ is a $C^{\infty}$ strictly
 plurisubharmonic exhaustion function on $D$.
Let $g^{D}$ and $g^{\Omega}$  be the K\"ahler metrics whose K\"ahler potentials are
$-\log\varphi(z)$ and $-\log [\varphi(z)-||z_{0}||^{2}]$ respectively.
If $g^{D}$ is a K\"ahler-Einstein metric,
then the following condition are equivalent:
\begin{description}
  \item[(i)] $g^{\Omega}$ is a K\"ahler-Einstein metric with Ricci curvature $-(d_{0}+d+1)$;
  \item[(ii)] $g^{\Omega}$ is an extremal metric;
  \item[(iii)]  The scalar curvature of $g^{\Omega}$ is a constant;
  \item[(iv)] The Ricci curvature of $g^{D}$  equals to $-(d+1)$.
\end{description}
\end{theorem}

Actually, we study this problem in a more general case.
The sufficient conditions for  $g^{\Omega}$  is
Einstein or extremal can be described by the curvatures respectively.
Thus we also provide some extremal K\"ahler metrics. But they are not Einstein.

\ms
The study of holomorphic isometric immersions (called K\"ahler immersions in the sequel)
 between K\"ahler manifolds was started by  Calabi \cite{Calabi1953}.
 He solved the problem of deciding about the
existence of K\"ahler immersions between K\"ahler manifolds and complex space forms.
If K\"ahler immersions exist,
the K\"ahler manifolds are also called  K\"ahler submanifolds of complex space forms.
Afterwards, there appear many important studies about the characterization and classification of K\"ahler submanifolds
of complex space forms. For example, Di Scala and Loi gave a complete description of the K\"ahler immersions of
Hermitian symmetric spaces into complex space forms in \cite{Di Scala2007}.
Di Scala, Ishi and Loi also studied the K\"ahler immersions of
homogeneous K\"ahler manifolds into complex space forms in \cite{Di Scala2012}.
In order to study the existence of  nonhomogeneous K\"ahler-Einstein  submanifolds of infinite dimensional complex projective forms,
 Loi and Zeddy studied the K\"ahler immersions of Cartan-Hartogs domains in  \cite{Loi2011}.
Inspired by their work, we study the K\"ahler immersions of
$(\Omega, g^{\Omega})$  into complex space forms.

 \begin{theorem}\label{thm:1.2}
Let $\Omega$  be as in \eqref{equ:Omega}. Suppose that $\Omega$ is
a simply connected circular domain with center zero and
the function $-\log\varphi$ is the special K\"ahler potential determined by the diastatic function of $g^{D}$.
Then the existence of  K\"ahler immersions of  $(\Omega, g^{\Omega})$
 into complex space forms are completely determined  by the existence of K\"ahler immersions of $(D, g^{D})$
 into complex space forms (see Table \ref{tab:1} of Section \ref{sec:3.5} for details).
\end{theorem}

\ms
The paper is organized as follows. In Section \ref{sec:2}, we
construct a natural K\"ahler metric on a
generalized bounded pseudoconvex Hartogs  domain.
The Ricci curvature and scalar curvature will be calculated directly by
the standard formulas.   Then we study the problem when this metric  is
Einstein or extremal. By our results in this section, Theorem \ref{thm:1.1} can be obtained immediately.
In Section \ref{sec:3}, after recalling some definitions and criterions for K\"ahler immersions, we discuss the existence
of  K\"ahler immersions of $(\Omega, g)$ into three types of complex space forms
respectively.

\section{Two canonical metrics of bounded pseudoconvex Hartogs domains}\label{sec:2}
In order to obtain more interesting properties, we prefer to study the bounded pseudoconvex Hartogs domain as follows:
\begin{equation}\label{equ:bounded pseudoconvex domain of Hartogs type}
\Omega=\Big\{(z_{0}, z)\in \mathbb{C}^{d_{0}}\times D : ||z_{0}||^{2}<\varphi(z)\Big\},
\end{equation}
where $D=D_{1}\times D_{2}\times \cdots \times D_{m}
\subset \mathbb{C}^{d_{1}}\times \mathbb{C}^{d_{2}}\times \cdots \times \mathbb{C}^{d_{m}}$, $m\in \mathbb{Z}^{+}$,
 is the product of finite bounded pseudoconvex domains,
and $\varphi=\prod_{i=1}^{m}\varphi_{i}$. Here $\varphi_{i}$ is a function on $D_{i}$ such that
$-\log \varphi_{i}$ is a $C^{\infty}$ strictly
 plurisubharmonic exhaustion function on $D_{i}$.
 Since $D_{i}$ is a bounded pseudoconvex domain, such $\varphi_{i}$  always exists.
Obviously, $-\log \varphi=-\sum_{i=1}^{m}\log \varphi_{i}$ is a  $C^{\infty}$ strictly
 plurisubharmonic exhaustion function on $D$.
Hartogs domain defined by   \eqref{equ:Hartogs domain}  can be obtained
from \eqref{equ:bounded pseudoconvex domain of Hartogs type} by taking $m=1$.

Let $F(z_{0},z)=||z_{0}||^2-\varphi(z)$.
The boundary $ \partial\Omega$ of $\Omega$ is consist of  two parts, i.e.
$$\partial\Omega=(\{0\}\times \partial D)\cup\partial_0\Omega,$$
 where $\{0\} \times \partial D =\{(0,z): z\in \partial D\}$
 and $\partial_0\Omega=\{(z_{0},z)\in \mathbb{C}^{d_{0}}\times D: F(z_{0},z)=0,z_{0}\neq0\}.$
Now we claim that $F$ is a local $C^{\infty}$ defining function of $\Omega$ at any fix boundary point
$\widetilde{p}=(\widetilde{z_{0}},\widetilde{z})\in\partial_0\Omega$.  In fact, let $V(\widetilde{z})\subset D$ be a
neighborhood of $\widetilde{z}$, $\mathbb{B}(\widetilde{z_{0}},r)$ be a
ball with radius $r<||\widetilde{z_{0}}||$. Then the neighborhood $U(\widetilde{p})=\mathbb{B}(\widetilde{z_{0}},r)  \times  V(\widetilde{z})$ of $\widetilde{p}$ satisfies
$$U(\widetilde{p})\cap \Omega=\left\{(z_{0},z)\in U(\widetilde{p}): F(z_{0},z)<0 \right\},$$ and
$dF(z_{0},z)\neq 0$ for $(z_{0},z)\in \partial_{0}\Omega.$
So the claim is true. Thus we know the boundary $\partial_0\Omega$ is always smooth.

\subsection{K\"ahler-Einstein metric}
Let $\Omega$ be as in \eqref{equ:bounded pseudoconvex domain of Hartogs type}.
Since $-\log\varphi_{i}$ is a $C^{\infty}$ strictly
 plurisubharmonic exhaustion function $\varphi_{i}$ on $D_{i}$,  it gives a global K\"ahler metric  on $D_{i}$, denoted by $g^{D_{i}}$.
 The K\"ahler form $\omega^{D_{i}}=\frac{\sqrt{-1}}{2}\partial \overline{\partial}(-\log\varphi_{i})$.
Hence, $(D_{i}, g^{D_{i}})$ is a  K\"ahler manifold. Let
\begin{equation}
(D, g^{D})=(D_{1}\times\cdots\times D_{m}, g^{D_{1}}\times\cdots\times g^{D_{m}}).
\end{equation}
If $g^{D_{i}}$, $i=1,2,\cdots,m$, are Einstein metrics  with the same Ricci curvature,
then $g^{D}$ is Einstein.

Notice that the function $-\log (-F)$ is a $C^{\infty}$
strictly plurisubharmonic function on $\Omega$, and $-\log (-F)\rightarrow \infty$ as $(z_{0}, z)\rightarrow \partial \Omega$.
It gives a global K\"ahler metric $g^{\Omega}$ of $\Omega$, i.e.
\begin{equation}\label{equ:Kahler form}
\omega^{\Omega}=\frac{\sqrt{-1}}{2}\partial \overline{\partial}(-\log (-F)).
\end{equation}

Let $(z_{0},z)=(z_{0},z_{1},z_{2},\cdots,z_{m})\in \mathbb{C}^{d_{0}}\times\mathbb{C}^{d_{1}}\times\cdots\times\mathbb{C}^{d_{m}}$ and
$d=\sum\limits_{i=1}^{m}d_{i}$,
where $z_{i}=(z_{i1},z_{i2},\cdots,z_{id_{i}})\in \mathbb{C}^{d_{i}}$.
In this coordinate, the  matrix of  $g^{\Omega}$ in \eqref{equ:Kahler form}, also denoted by the same notation,
 can be written as follows:
\begin{equation}\label{equ:g Omega}
g^{\Omega}=\left(g^{\Omega}_{j\alpha,\overline{k\beta}}\right),
\end{equation}
where the elements
$$
g^{\Omega}_{j\alpha, \overline{k\beta}}=-\frac{\partial^{2}\log \left(\varphi(z)-||z_{0}||^{2}\right)}{\partial z_{j\alpha}\partial \overline{z}_{k\beta}} \quad \text{for} ~0\leq j, k\leq m ~\text{and}~ 1\leq\alpha\leq d_{j}, 1\leq\beta\leq d_{k}.
$$

\begin{lemma}\label{lem:2.1}
Let $\Omega$ be as in \eqref{equ:bounded pseudoconvex domain of Hartogs type}.
If $g^{D_{i}}, i=1,\cdots,m,$ are K\"ahler-Einstein metrics with Ricci curvatures $c_{i}$ respectively,
then there exist some real-valued functions $ f_{i}$ on $D_{i}$ respectively,
such that
\begin{equation}
\det (g^{\Omega})=(\varphi-||z_{0}||^{2})^{-(d+d_{0}+1)}\prod_{i=1}^{m}\varphi_{i}^{d+1+c_{i}}e^{c_{i}f_{i}}.
\end{equation}
\end{lemma}

\begin{proof} For convenient, we define
 $$ \varphi_{j\alpha}:=\frac{\partial \varphi(z)}{\partial z_{j\alpha}},
 ~\varphi_{j\alpha,\overline{k\beta}}:=\frac{\partial^{2} \varphi(z)}{\partial z_{j\alpha}\partial \overline{z}_{k\beta}} ~\text{for} ~1\leq j, k\leq m.$$
By a straightforward computation, the metric
\begin{equation}\label{equ:Hession of solution}
g^{\Omega}=\frac{1}{(\varphi-||z_{0}||^{2})^{2}}\left(
  \begin{array}{c|c}
  & \\
(\varphi-||z_{0}||^{2})\delta_{st}+ \overline{z}_{0s}z_{0t}& -\overline{z}_{0s} \varphi_{\overline{k\beta}} \\
& \\
\hline
& \\
 -\varphi_{j\alpha}z_{0t} &\varphi_{j\alpha} \varphi_{\overline{k\beta}}-\varphi_{j\alpha,\overline{k\beta}}(\varphi-||z_{0}||^{2})  \\
 & \\
\end{array}
\right),
\end{equation}
where
the upper left block is $(\varphi-||z_{0}||^{2})I^{(d_{0})}+ \overline{z}_{0}^{t}z_{0}$
and the downer right block is a $d\times d$ submatrix.
Now we should make some elementary determinant calculations.
Indeed,  fix $1\leq s \leq m$, we know that
\begin{equation}\label{equ:sum1}
\sum_{t=1}^{m}((\varphi-||z_{0}||^{2})\delta_{st}+ \overline{z}_{0s}z_{0t})\frac{\overline{z}_{0t}\varphi_{\overline{k\beta}}}{\varphi}
=\overline{z}_{0s} \varphi_{\overline{k\beta}},
\sum_{t=1}^{m} -\varphi_{j\alpha}z_{0t} \frac{\overline{z}_{0t}\varphi_{\overline{k\beta}}}{\varphi}
=-\frac{\varphi_{j\alpha}\varphi_{\overline{k\beta}}||z_{0}||^{2}}{\varphi}.
\end{equation}
Under the elementary transformations above, the matrix can be transformed into
\begin{equation}\label{equ:Hession of solution 1}
\frac{1}{(\varphi-||z_{0}||^{2})^{2}}\left(
  \begin{array}{c|c}
     & \\
  (\varphi-||z_{0}||^{2})\delta_{st}+ \overline{z}_{0s}z_{0t} & 0 \\
     & \\
    \hline
  & \\
   -\varphi_{j\alpha}z_{0t}  &   (\varphi-||z_{0}||^{2})\varphi\frac{\varphi_{j\alpha} \varphi_{\overline{k\beta}}-\varphi_{j\alpha,\overline{k\beta}}\varphi}{\varphi^{2}}  \\
     & \\
  \end{array}
\right).
\end{equation}

Let $g^{D_{i}}_{\alpha\overline{\beta}}$ be the $(\alpha,\beta)$-entry of the K\"ahler-Einstein metric $g^{D_{i}}$,
i.e.
\begin{equation}\label{equ:gDii}
g^{D_{i}}_{\alpha\overline{\beta}}
=\frac{\partial^{2} (-\log \varphi_{i})}{\partial z_{i\alpha}\partial \overline{z}_{i\beta}}(z_{i},z_{i}).
\end{equation}
Since
\begin{equation}\label{equ:gDi}
\frac{\varphi_{j\alpha} \varphi_{\overline{k\beta}}-\varphi_{j\alpha,\overline{k\beta}}\varphi}{\varphi^{2}}=\left\{
\begin{array}{l}
g^{D_{j}}_{\alpha\overline{\beta}} \ \ \  \ \  \ \ \ \ \ \ \ \ \ j=k; \\
  0 \ \ \ \ \ \ \ \ \ \ \ \ \ \ \ \  \ \    j\neq k.
\end{array}
\right.
\end{equation}
This implies that the downer right block of \eqref{equ:Hession of solution 1} is a block diagonal matrix.
If $g^{D_{i}}$ is K\"ahler-Einstein, then
there exists a real-valued pluriharmonic function  $f_{i}$ on $D_{i}$ such that
\begin{equation}\label{equ:det of gDi}
\det g^{D_{i}}=\varphi_{i}^{c_{i}}e^{c_{i}f_{i}}.
\end{equation}
In order to obtain the determinant of $g^{\Omega}$, we recall a well known formulation:
$$\det(I^{(p)}+A\overline{B}^{t})=\det(I^{(q)}+\overline{B}^{t}A),$$
where $A$ is a $p\times q$ matrix, $B$ is a $q\times p$ matrix.
So we can get
\begin{equation}\label{equ:det of T2}
\det((\varphi-||z_{0}||^{2})I^{(d_{0})} +\overline{z}_{0}^{t}z_{0})=(\varphi-||z_{0}||^{2})^{d_{0}}(1+\frac{||z_{0}||^{2}}{\varphi-||z_{0}||^{2}})=(\varphi-||z_{0}||^{2})^{d_{0}-1}\varphi.
\end{equation}
Form \eqref{equ:gDi} and \eqref{equ:det of T2}, it follows that
$$\det (g^{\Omega})=\frac{\varphi^{d+1}}{(\varphi-||z_{0}||^{2})^{d+d_{0}+1}}\prod_{i=1}^{m}\det (g^{D_{i}}_{\alpha\overline{\beta}}).$$
Then, by \eqref{equ:det of gDi} we obtain
\begin{equation}
\det (g^{\Omega})=\frac{1}{(\varphi-||z_{0}||^{2})^{d+d_{0}+1}}\prod_{i=1}^{m}\varphi_{i}^{d+1+c_{i}}e^{c_{i}f_{i}}.
\end{equation}
We complete the proof.
\end{proof}

\ms
By using the standard formula of Ricci tensor, i.e.
$$\mathrm{Ric}_{j\alpha,\overline{k\beta}}=-\frac{\partial^{2}\log \det g^{\Omega}}{\partial z_{j\alpha} \partial \overline{z}_{k\beta} },$$
we can obtain the following lemma  directly.
\begin{lemma}\label{lem:Ricci}
Let $\mathrm{Ric}_{g}$ be the $\mathrm{Ricci}$ tensor  of $g^{\Omega}$. Then
\begin{equation}
\mathrm{Ric}_{g}=\left(
  \begin{array}{cccc}
      0 &0         & \cdots & 0  \\
    0 & \lambda_{1}g^{D_{1}}         & \cdots  & 0        \\
    \vdots &  \vdots                       & \ddots  &\vdots    \\
    0 &          0                 &\cdots   & \lambda_{m}g^{D_{m}}  \\
  \end{array}
\right)-(d+d_{0}+1)g^{\Omega},
\end{equation}
where  $\lambda_{i}=d+1+c_{i}$. Here $c_{i}$  denotes the Ricci curvature of $g^{D_{i}}$.
\end{lemma}
\begin{theorem}\label{thm:2.1}
Let $\Omega$ be as in \eqref{equ:bounded pseudoconvex domain of Hartogs type}.
If $g^{D_{i}}$, $1\leq i \leq m$, are K\"ahler-Einstein metrics with Ricci curvatures $c_{i}$ respectively,
then $g^{\Omega}$ is  Einstein
if and only if
$c_{i}=-(d+1)$ for all $1\leq i \leq m$.
\end{theorem}

\subsection{Extremal metric}
In the following, we will try to obtain the criterion that  $g^{\Omega}$ is extremal.
Let $(M, g)$ be a $n$-dimensional  K\"ahler  manifold, and
 $(z_1,...,z_{n})$ be the local coordinate in  a neighbourhood of $p\in M$.
Let $s_{g}$ the scalar curvature of $g$, from Calabi's result in \cite{Calabi1982},
the extremal  condition can be given by the following equations:
\begin{equation}\label{equ:extremal condition}
\frac{\partial}{\partial \overline{z}_{\eta}}\sum_{\beta=1}^{n}g^{\beta\overline{\alpha}}
\frac{\partial s_{g}}{\partial \overline{z}_{\beta}}=0,
\end{equation}
for all $\alpha, \eta=1,\cdots, n$.

\ms
Now we calculate the scalar curvature of $g^{\Omega}$ firstly.
Let
 \begin{equation}\label{equ:g-1 Omega}
g_{\Omega}=\left(  g_{\Omega}^{j\alpha,\overline{k\beta}} \right)
\end{equation}
be the inverse matrix of $g^{\Omega}$, where $0\leq j, k\leq m$ and $1\leq\alpha\leq d_{j}$, $1\leq\beta\leq d_{k}$.
Let $g_{D_{i}}=(g_{D_{i}}^{\alpha\overline{\beta}})$ be the inverse matrix of
$g^{D_{i}}=(g^{D_{i}}_{\alpha\overline{\beta}})$ in \eqref{equ:gDii}. By the same method in Lemma \ref{lem:2.1}, a direct computation implies
\begin{equation}
g_{\Omega}^{j\alpha,\overline{k\beta}}=\left\{
\begin{array}{l}
\frac{\varphi-||z_{0}||^{2}}{\varphi}g_{D_{j}}^{\alpha\overline{\beta}} \ \ \  \   \ \ \ \  \ \ \ j=k; \\
  0 \ \ \ \ \ \ \ \ \ \ \ \ \ \ \ \  \ \ \ \ \ \ \ \ \  \ \ \ \   j\neq k,
\end{array}
\right.
\end{equation}
where $j\neq0, k\neq0$.
This shows that the downer right block of \eqref{equ:g-1 Omega} is a $d\times d$ block diagonal matrix.
\begin{lemma}\label{lem:scalar}
  Let $s_{g^{\Omega}}$ be the scalar curvature of $g^{\Omega}$, then
$$s_{g^{\Omega}}=\frac{\tau(\varphi-||z_{0}||^{2})}{\varphi}-(n+1)n,$$
where $\tau=(d+1)d+\sum_{i=1}^{m}c_{i}d_{i}$.
Furthermore, $s_{g^{\Omega}}$ is  constant if  and only if $\tau=0$.
\end{lemma}
\begin{proof}
 According to the formula of the scalar curvature, it follows that
\begin{eqnarray*}
s_{g^{\Omega}}
&=&\sum_{i=1}^{m}\sum_{i\alpha,i\beta=i1}^{id_{i}}\lambda_{i}\frac{\varphi-||z_{0}||^{2}}{\varphi}
g_{D_{i}}^{\beta\overline{\alpha}}g^{D_{i}}_{\alpha\overline{\beta}}-(d+d_{0}+1)(d+d_{0})\\
&=&\sum_{i=1}^{m}\lambda_{i}d_{i}\frac{\varphi-||z_{0}||^{2}}{\varphi}-(d+d_{0}+1)(d+d_{0}).
\end{eqnarray*}
where $\lambda_{i}=(d+1+c_{i})$, and $c_{i}$  is the Ricci curvature of $g^{D_{i}}$.  Let $\tau=\sum_{i=1}^{m}\lambda_{i}d_{i}$, then
$\tau=(d+1)d+\sum_{i=1}^{m}c_{i}d_{i}$. Hence,  we can complete the proof.
\end{proof}

\begin{theorem}\label{thm:2.2}
Let $\Omega$ be as in \eqref{equ:bounded pseudoconvex domain of Hartogs type}
If $g^{D_{i}}$ is a K\"ahler-Einstein metric with Ricci curvature $c_{i}$,
 then
$g^{\Omega}$ in \eqref{equ:Kahler form} is extremal if and only if
its  scalar curvature $s_{g^{\Omega}}$ is a constant.
\end{theorem}
\begin{proof}
By Calabi's result, it suffices  to show  that
$g^{\Omega}$ is not an extremal metric if $\tau\neq 0$.
By \eqref{equ:Hession of solution}, for any fix $0\leq t \leq d_{0}$ and $1\leq (k\beta) \leq d_{k}$, $k\neq0$, we know
\begin{equation}
z_{0t}=\frac{(\varphi-||z_{0}||^{2})^{2}}{\varphi}\sum_{s=1}^{m}g^{\Omega}_{0s,\overline{0t}}z_{0s} \ \text{and}
\ \
||z_{0}||^{2}\varphi_{\overline{k\beta}}
=-(\varphi-||z_{0}||^{2})^{2}\sum_{s=1}^{m}g^{\Omega}_{0s,\overline{k\beta}} z_{0s}, k\neq 0.
\end{equation}
By Lemma \ref{lem:scalar} and the equations above, we can obtain
$$\frac{\partial s_{g}}{\partial \overline{z}_{0t}}=-\frac{\tau z_{0t}}{\varphi}=-\frac{\tau(\varphi-||z_{0}||^{2})^{2}}{\varphi^{2}}\sum_{s=1}^{m}g^{\Omega}_{0s,\overline{0t}}z_{0s},$$
$$\frac{\partial s_{g}}{\partial \overline{z}_{k\beta}}=\frac{\tau||z_{0}||^{2}\varphi_{\overline{k\beta}}}{\varphi^{2}}
=-\frac{\tau(\varphi-||z_{0}||^{2})^{2}}{\varphi^{2}}\sum_{s=1}^{m}g^{\Omega}_{0s,\overline{k\beta}} z_{0s}, k\neq 0.$$
Hence, we have
\begin{eqnarray*}
&& \sum_{t=1}^{d_{0}}g_{\Omega}^{0t,\overline{01}}\frac{\partial s_{g}}{\partial\overline{z}_{0t}}+\sum_{k=1}^{m}\sum_{\beta=1}^{d_{k}} g_{\Omega}^{k\beta,\overline{01}}\frac{\partial s_{g}}{\partial\overline{z}_{k\beta}}
\\
&=&-\frac{\tau(\varphi-||z_{0}||^{2})^{2}}{\varphi^{2}}(\sum_{t=1}^{d_{0}}g_{\Omega}^{0t,\overline{01}}\sum_{s=1}^{m}g^{\Omega}_{0s,\overline{0t}}z_{0s}
+\sum_{k=1}^{m}\sum_{\beta=1}^{d_{k}} g_{\Omega}^{k\beta,\overline{01}}\sum_{s=1}^{m}g^{\Omega}_{0s,\overline{k\beta}} z_{0s})
\\
&=&-\frac{\tau z_{01}(\varphi-||z_{0}||^{2})^{2}}{\varphi^{2}}.\\
\end{eqnarray*}
By comparing it with Calabi's extremal condition \eqref{equ:extremal condition}, we know that  $g^{\Omega}$ is not extremal if $\tau\neq0$.
Consequently, we complete the proof.
\end{proof}
\ms
Immediately, Theorem \ref{thm:1.1} will be a direct consequence by taking $m=1$.

\begin{proof}[The proof of Theorem \ref{thm:1.1}]
By Theorem \ref{thm:2.1}, (i)$\Leftrightarrow$(iv). By Lemma \ref{lem:Ricci}, Lemma \ref{lem:scalar},  (iv)$\Leftrightarrow$(iii).
By Theorem \ref{thm:2.2}, (iii)$\Leftrightarrow$(ii).
\end{proof}

\subsection{Some classical examples}
Now we will show some classical examples of bounded pseudoconvex Hartogs domains as an application of our results in previous section.
Recall some basic notions firstly. Let $\mathrm{Aut}(D)$  be the group of automorphisms of  a bounded domain $D\subset \mathbb{C}^{d}$.
 $D$ is called homogeneous if there exists a map $\Phi\in \mathrm{Aut}(\Omega)$ such that $\Phi(a)=b$ for two arbitrary points $a, b\in D$.
Moreover, $D$ is called symmetric if for every point $a\in D$, there exists an automorphism $\Phi\in \mathrm{Aut}(\Omega)$ such that  $\Phi(a)=a$, $\Phi\circ\Phi=id$,
and $a$ is an isolated point of the set $\{z\in D: \Phi(z)=z\}$.
The irreducible bounded symmetric domains have been
completely classified up to a biholomorphic isomorphism due to E. Cartan \cite{Cartan1935}.
Each bounded symmetric domain is biholomorphic to a Cartesian product of domains belongs to the following six Cartan types.

Let $M_{m,n}$ be the space of $m\times n$-complex matrices, $I$ be the unit $p\times p$-matrix,
  $\bar{z}$ be the conjugate matrix of $z$, $z^{t}$ be the transposed matrix of $z$.
  If a square matrix  $A$ is positive definite, then we denote it by $A>0$.
  The list of irreducible bounded symmetric domains and the corresponding generic norms is as follows.
  \begin{description}
  \item [Type I] $(1\leq m\leq n)$: \quad \  \ \  $D_{\mathrm{I}}=\left\{z\in M_{m,n}(\mathbb{C}): I-z\bar{z}^{t}>0\right\}$,
   $N(z,w)=\det(I-z\overline{w}^{t})$.
  \item [Type II] $(m=n\geq 5)$:  \ \quad  $D_{\mathrm{II}}=\big\{z\in D_{\mathrm{I}}: z=-z^{t}\big\}$,
  \ \ \quad \quad \quad \ \  $N(z,w)=\det(I+z\overline{w})$.
  \item [Type III]  $(m=n\geq 2)$:  \ \ \  $D_{\mathrm{III}}=\big\{z\in D_{\mathrm{I}}: z=z^{t}\big\}$,
  \ \quad \quad \quad \quad \  $N(z,w)=\det(I-z\overline{w}).$
  \item [Type IV] $ (m\geq 5)$: \ \ \ \quad \quad  $D_{\mathrm{IV}}=\big\{z\in\mathbb{C}^{m}: 1-2q(z,\bar{z})+|q(z,z)|^{2}>0,
   |q(z,\bar{z})|<1\big\}$,
   $$ N(z,w)=1-q(z,w)+q(z,z)q(w,w), ~\text{where}~~ q(z,w)=\sum\limits_{j=1}^{m}z_{j}w_{j}.$$
  \item [Type V:] \ $D_{\mathrm{V}}=\big\{z\in M_{2,1}(O_{\mathbb{C}}):1-(z|z)+(z^{\sharp}|z^{\sharp})>0, 2-(z|z)>0\big\},$
  $$N(z,w)=1-(z|w)+(z^{\sharp}|w^{\sharp}).$$
\item [Type VI:]
     $D_{\mathrm{VI}}=\big\{z\in M_{3,3}(O_{\mathbb{C}}):1-(z|z)+(z^{\sharp}|z^{\sharp})-|\det z|^{2}>0$,
     $3-2(z|z)+(z^{\sharp}|z^{\sharp})>0,3-(z|z)>0\big\},$
  $$N(z,w)=1-(z|w)+(z^{\sharp}|w^{\sharp})-\det z \overline{\det w}.$$
where $O_{\mathbb{C}}=\mathbb{C}\otimes \mathbb{O}$ is complex 8 dimensional Cayley algebra.
$M_{3,3}(O_{\mathbb{C}})$ is the space of $3\times3$ matrices with entries in the space $O_{\mathbb{C}}$ of octonions over $\mathbb{C}$,
 which are Hermitian with respect to the Cayley conjugation. $z^{\sharp}$ is the adjoint matrix in
$M_{3,3}(O_{\mathbb{C}})$ and $(z|w)$ is the standard Hermitian product in $M_{3,3}(O_{\mathbb{C}})$, $M_{2,1}(O_{\mathbb{C}})$ is a
 subspace of $M_{3,3}(O_{\mathbb{C}})$.
  \end{description}

\ms
The domains of types $\mathrm{I-IV}$  are classical,
$D_{\mathrm{V}}$ and  $D_{\mathrm{VI}}$ are the exceptional 16 and 27 dimensional domains.
 These domains are also called Cartan domains.
Let $g^{D}_{B}$, $K(z,z)$, $r$, $\gamma$, $V(D)$  be the Bergman metric,
Bergman kernel, rank, genus, volume  of Cartan domain $D$ respectively, then
\begin{equation}\label{equ:det of gD}
\det g^{D}_{B}=\gamma^{d}V(D)K(z,z).
\end{equation}
Obviously, the Ricci curvature is $-1$.
This result is also true for any bounded homogeneous domain.
For Cartan domain,  the connection between the generic norm $N(z,z)$
and the Bergman kernel $K(z,z)$
 is
\begin{equation}\label{equ:the connection of N and K}
V(D)K(z,z)=N(z,z)^{-\gamma}.
\end{equation}
The Wallach set can be described as follow.
\begin{equation}\label{equ:wallach set}
  W(D)=\{0,\frac{a}{2},\cdots,(r-1)\frac{a}{2}\}\cup((r-1)\frac{a}{2},\infty),
\end{equation}
where  $a, b$ are invariant numbers. For more details we refer the reader to  \cite{Arazy1995}.

\ms
\begin{example}
\label{ex:CH}
Let $ D\subset \mathbb{C}^{d} $ be a Cartan domain.
Let $g^{D}$ be the  metric generated by the function $-\mu\log N(z,z)$, $\mu>0$.
 Thus $g^{D}=\frac{\mu}{\gamma}g^{D}_{B}$ is a K\"ahler-Einstein metric with Ricci curvature $-\frac{\gamma}{\mu}$.
Now, we take
$\varphi(z)=N(z,z)^{\mu}$, then we get the so-called Cartan-Hartogs domain introduced by Yin and Roos \cite{Wang2006}:
  \begin{equation}\label{equ:C-H}
\Omega_{CH}=\Big\{(z_{0}, z)\in \mathbb{C}^{d_{0}}\times D : ||z_{0}||^{2}<N(z,z)^{\mu}, \mu >0 \Big\}.
\end{equation}
It is homogeneous if and only if $\Omega_{CH}$ is a ball, i.e. $D=\mathbb{B}^{d}$ and $\mu=1$ (see \cite{Ahn2012}).
Let
$g^{\Omega_{CH}}$ be the metric generated by $-\log(N(z,z)^{\mu}-||z_{0}||^{2})$.
By Theorem \ref{thm:1.1}, $g^{\Omega_{CH}}$ is extremal if and only if it is K\"ahler-Einstein, i.e. $\mu=\frac{\gamma}{d+1}$.
This is coincide with Loi and Zedda's results in \cite{Zedda2012}.
Now we consider a Hartogs domain over Cartan-Hartogs domain.
  \begin{equation}
\widetilde{\Omega}_{CH}=\Big\{(\widetilde{z}_{0},z_{0}, z)\in \mathbb{C}^{\widetilde{d}_{0}}\times \Omega_{CH} : ||\widetilde{z}_{0}||^{2}<\varphi(z)\Big\}.
\end{equation}
where $-\log\varphi(z)=-\log(N(z,z)^{\mu}-||z_{0}||^{2})$ is a strictly plurisubharmonic exhaustion function on $\Omega_{CH}$.
Hence,
$-\log (N(z,z)^{\mu}-||z_{0}||^{2}-||\widetilde{z}_{0}||^{2})$  generates a complete K\"ahler-Einstein metric with Ricci curvature $-(d+d_{0}+\widetilde{d}_{0}+1)$
if $\mu=\frac{\gamma}{d+1}$.
In fact, $\widetilde{\Omega}_{CH}$ is also a Cartan-Hartogs domain with $(\widetilde{d}_{0}+d_{0})$-dimensional fibers, i.e.
  \begin{equation}
\widetilde{\Omega}_{CH}=\Big\{(\widetilde{z}_{0},z_{0}, z)\in \mathbb{C}^{\widetilde{d}_{0}}\times \mathbb{C}^{d}_{0}\times D: ||\widetilde{z}_{0}||^{2}+||z_{0}||^{2}<N(z,z)^{\mu}\Big\}.
\end{equation}
\end{example}

\ms
\begin{example}
Let $ D\subset \mathbb{C}^{d} $ be  a bounded homogeneous domain,
 $K(z,z)$ be the Bergman kernel, $\mathrm{V}(D)$ be the volume of $D$.
Let $g^{D}$ be the metric generated by $\log K(z,z)^{\nu}$ with $\nu>0$.
Thus $g^{D}=\nu g^{D}_{B}$ is a K\"ahler-Einstein metric with the Ricci curvature $-\frac{1}{\nu}$.
We take
$\varphi(z)=K(z,z)^{-\nu}$, then we get the so-called Bergman-Hartogs domain:
  \begin{equation}\label{equ:B-H}
\Omega_{BH}=\Big\{(z_{0}, z)\in \mathbb{C}^{d_{0}}\times D : ||z_{0}||^{2}<K(z,z)^{-\nu}, \nu >0 \Big\}.
\end{equation}
 Let $g^{\Omega_{BH}}$ be the metric generated by $-\log(K(z,z)^{-\nu}-||z_{0}||^{2})$.
Then $g^{\Omega_{BH}}$ is extremal if and only if it is K\"ahler-Einstein, i.e. $\nu=\frac{1}{d+1}$.

\ms
\end{example}
\begin{example} Fock-Bargmann-Hartogs domain:
  \begin{equation}\label{equ:FBH}
\Omega_{FBH}=\Big\{(z_{0}, z)\in \mathbb{C}^{d_{0}}\times \mathbb{C}^{d} : ||z_{0}||^{2}<e^{-\mu||z||^{2}}, \mu>0 \Big\}.
\end{equation}
Although this domain is unbounded, our result  is also valid.
The metric of $\mathbb{C}^{d}$  given by $-\log \varphi=\mu||z||^{2}$
is flat, thus the K\"ahler metric whose K\"ahler potential is $-\log(e^{-\mu||z||^{2}}-||z_{0}||^{2})$ is not Einstein.
\end{example}

\section{K\"ahler immersions  in complex space forms}\label{sec:3}
In this section, we will study the existence of K\"ahler immersions of
bounded pseudoconvex Hartogs domains  into complex space forms.
The complex space forms are K\"ahler manifolds of constant holomorphic
sectional curvatures. Assume that they are complete and simply connected.
According to the sign of the constant holomorphic sectional curvature,
there are three types:
\begin{description}
  \item[(1)] Complex Euclidean space $(\mathbb{C}^{N}, g_{0}), N\leq +\infty$, where
$g_{0}$ denotes the flat metric. Here $\mathbb{C}^{\infty}$ is the complex Hilbert space $\ell^{2}(\mathbb{C})$ consisting
of sequences $z_{j}\in \mathbb{C},$ $j = 1,\cdots, $ such that $\sum_{j=1}^{\infty}|z_{j}|^{2} < +\infty$.
  \item[(2)] Complex hyperbolic space $\mathbb{CH}^{N}, N\leq +\infty,$ namely the unit ball
in $\mathbb{C}^{N}$, $\sum_{j=1}^{N}|z_{j}|^{2} < +\infty$ endowed with the hyperbolic metric $g_{hyp}$ of
holomorphic sectional curvature being $-4$, whose associated K\"ahler form is
$$\omega_{hyp}=-\frac{\sqrt{-1}}{2} \partial \overline{\partial}\log\sum_{j=1}^{N}(1-|z_{j}|^{2}).$$
  \item[(3)] Complex projective space $\mathbb{CP}^{N}, N\leq +\infty$, with the Fubini-Study
metric $g_{FS}$ of holomorphic sectional curvature being $4$. If $\omega$ denotes
the K\"ahler form associated to $g_{FS}$. Let $[Z_{0}, \cdots,Z_{N}]$ be the  homogeneous coordinates
in  $\mathbb{CP}^{N}$. In the affine chart $U_{0}=\{Z_{0}\neq 0\}$ endowed with coordinates $(z_{1},\cdots,z_{N})$, $z_{j}=\frac{Z_{j}}{Z_{0}}$.
 $$\omega_{FS}=-\frac{\sqrt{-1}}{2} \partial \overline{\partial}\log\sum_{j=1}^{N}(1+|z_{j}|^{2}).$$

\end{description}
\subsection{Some criterions}

A fundamental criterion is Calabi's criterion which is based on the \emph{diastatic function}.
Let us introduce the definition firstly.
Let $(M,~g)$ be a $n$-dimensional K\"ahler manifold. In a local complex coordinate $(U, z)$,
the K\"ahler form  $\omega=\frac{\sqrt{-1}}{2}\partial \overline{\partial}\Phi$,
where $\Phi$ is a   local K\"ahler potential.
If $(M, g)$  admits a K\"ahler immersion into  a complex space from,
Calabi proved that $g$ must be real analytic
(see Theorem 8 in \cite{Calabi1953}). In this case,
the K\"ahler potential $\Phi(z)$ can be expressed as a power series in the $2n$ real
variables in $U$.
Substitute  $\Phi(z,\overline{z})$ for $\Phi(z)$.
Let $p$ and $q$ denote two arbitrary points in $U$,
 $z(p)$ and $z(q)$ denote the  local complex coordinate of  $p$ and $q$ respectively,
then $\Phi(z(p),\overline{z(q)})$ is a real analytic function in $U\times \overline{U}$.
Then Calabi introduced the \emph{diastatic function}:
\begin{equation}\label{equ:diastatic}
D(p,q)=\Phi(z(p),\overline{z(p)})+\Phi(z(q),\overline{z(q)})-\Phi(z(p),\overline{z(q)})-\Phi(z(q),\overline{z(p)}).
\end{equation}

One of the elementary properties  of the diastatic function is that
it is uniquely determined by the K\"ahler metric and independent of the local K\"ahler potential function
(see \cite{Calabi1953}).
Obviously, $D(p,q)=D(q,p)$.
Suppose that the origin $o \in U$, then $D(o,q)$ is a special K\"ahler potential that determined by
the diastatic function.
In a neighbourhood of the origin $o$, the power series of $D(o,q)$  is
\begin{equation}\label{equ:D a}
\displaystyle{D_{o}(q)=D(o,q)=\sum_{\alpha,\beta\geq0} a_{\alpha\beta}z^{\alpha}\overline{z}^{\beta}},
\end{equation}
where the multi-indexes $\displaystyle{\alpha=(\alpha_{1},\cdots,\alpha_{n}),}$
$\beta=(\beta_{1},\cdots,\beta_{n}),$
$z^{\alpha}=\prod\limits_{j=1}^{n}(z_{j})^{\alpha_{j}},
\overline{z}^{\beta}=\prod\limits_{j=1}^{n}(\overline{z}_{j})^{\beta_{j}}$.

Now define an ordering for the set of multi-indexs.
Consider  two arbitrary different multi-indexs $\alpha=(\alpha_{1},\alpha_{2},\cdots,\alpha_{n})$,
$\widetilde{\alpha}=(\widetilde{\alpha}_{1},\widetilde{\alpha}_{2},\cdots,\widetilde{\alpha}_{n})$.
\begin{description}
  \item[(1)] If $|\alpha|\neq|\widetilde{\alpha}|$, then
$\alpha>\widetilde{\alpha}\Leftrightarrow  |\alpha|>|\widetilde{\alpha}|$;
  \item[(2)] If $|\alpha|=|\widetilde{\alpha}|$, then
$\alpha>\widetilde{\alpha}\Leftrightarrow  \exists ~1\leq j\leq n$ such that
$\alpha_{j}<\widetilde{\alpha}_{j}$, and $\alpha_{i}=\widetilde{\alpha}_{i}$ for any $1\leq i\leq j-1$.
\end{description}
Under the ordering, the multi-indexes can be denoted by
$m_{0}=(0,0\cdots,0)$,  $m_{1}=(1,0\cdots,0)$, $\cdots$, $m_{n}=(0,\cdots,0,1),\cdots$.
The diastatic function can be described  in terms of Bochner coordinate and
$$D_{o}(q)=D(o,q)=\sum\limits_{j,k=0}^{\infty} a_{j,k}(z)^{m_{j}}(\overline{z})^{m_{k}},$$
where the multi-index $m_{j}=(m_{j,1},m_{j,2},\cdots,m_{j,n})$, $m_{k}=(m_{k,1},m_{k,2},\cdots,m_{k,n})$,
$|m_{j}|=|m_{j,1}+m_{j,2}+\cdots+m_{j,n}|,$
$|m_{k}|=|m_{k,1}+m_{k,2}+\cdots+m_{k,n}|$.

On this complex manifold $M$, Calabi defined two new K\"ahler metrics
whose K\"ahler potentials are  $e^{D_{o}}-1$, $1-e^{-D_{o}}$ respectively.
 Their power series can be written as follows.
\begin{equation}\label{equ:D bc}
e^{D_{o}}-1= \sum_{j,k\geq0} b_{jk}(g) z^{m_{j}} \overline{z}^{m_{k}}, \quad
 1-e^{-D_{o}}= \sum_{j,k\geq0} c_{jk}(g)z^{m_{j}} \overline{z}^{m_{k}}.
 \end{equation}

Let us introduce two definitions before giving Calabi's criterion.
\begin{definition}[Calabi  \cite{Calabi1953}]
  A K\"ahler immersion $f$ of $(M, g)$ into $\mathbb{C}^{N}$ (resp. $\mathbb{CP}^{N}$ or $\mathbb{CH}^{N}$)
   is said to be full if $f(M)$ can not be contained in any complex totally geodesic
   hypersurface of $\mathbb{C}^{N}$ (resp. $\mathbb{CP}^{N}$ or $\mathbb{CH}^{N}$).
\end{definition}

\begin{definition}[Calabi  \cite{Calabi1953}]
 The K\"ahler metric $g$ on a complex manifold $M$ is resolvable (resp.  $1$-resolvable or $-1$-resolvable) of
rank $N$ at $p$ if the $\infty\times\infty$ matrix $a_{jk}(g)$ (resp. $b_{jk}(g)$ or $c_{jk}(g)$) given by formula
\eqref{equ:D a}(resp. \eqref{equ:D bc}) is positive semidefinite and of
rank $N$,  $N\leq\infty$.
\end{definition}

In \cite{Di Scala2007}, Di Scala and Loi reorganized Calabi's results as follows.
\begin{theorem} [Calabi's criterion]
 Let $M$ be a complex manifold endowed with a real analytic K\"ahler
metric $g$.
\begin{description}
  \item[(i)] If g is resolvable (resp. $1$-resolvable or $-1$-resolvable) of rank $N$ at $p \in M$ then it is
resolvable (resp. $1$-resolvable or $-1$-resolvable) of rank $N$ at every point in $M$.
  \item[(ii)] A neighborhood of a point $p$ admits a (full) K\"ahler immersion into $\mathbb{C}^{N}$ (resp. $\mathbb{CP}^{N}$ or $\mathbb{CH}^{N}$) if and only if $g$ is resolvable (resp. $1$-resolvable or $-1$-resolvable) of rank
at most (exactly) $N$ at $p$.
  \item[(iii)]  Two full K\"ahler immersions into $\mathbb{C}^{N}$ (resp. $\mathbb{CP}^{N}$ or $\mathbb{CH}^{N}$) are congruent under the
isometry group of $\mathbb{C}^{N}$ (resp. $\mathbb{CP}^{N}$ or $\mathbb{CH}^{N}$).
\end{description}
\end{theorem}

\begin{theorem}[Calabi's criterion]\label{Calabi's criterion}
 Let $(M, g)$ be a simply connected K\"ahler manifold. If a neighborhood
of a point $p \in M$  can be K\"ahler immersed into a complex space form $(S, G)$ then the whole
$(M, g)$ admits a K\"ahler immersion into $(S, G)$.
\end{theorem}

\ms
Besides Calabi's criterion, we also need the follow theorems about
the relations between the K\"ahler submanifolds of complex space forms.

\begin{lemma}[Umehara \cite{Umehara1987}]\label{lem:Umehara}
If a K\"ahler manifold $(M, g)$  admits a K\"ahler immersion into $\mathbb{C}^{N}$, $N<\infty$, then
it can not be K\"ahler immersed into any finite dimensional complex hyperbolic
space or complex projective space.
If it can be K\"ahler immersed into $(\mathbb{CH}^{N}, g_{hyp}), N<\infty$, then it can not be into any
finite dimensional complex projective space.
\end{lemma}

\begin{lemma}[Di Scala, Ishi, Loi \cite{Di Scala2012}]\label{lem:Di2012}
If a K\"ahler manifold $(M, g)$  admits a K\"ahler immersion into $(\mathbb{CH}^{N}, g_{hyp})$, $N\leq \infty$,
then it also can be K\"ahler immersed into $(\mathbb{C}^{\infty}, g_{0})$.
\end{lemma}

\begin{lemma}[Zedda \cite{Zedda2009}]\label{lem:zedda2009}
A K\"ahler manifold $(M, hg^{M})$ admits a local K\"ahler immersion into $(\mathbb{CP}^{\infty}, g_{FS})$ for all $h>0$ if
 and only if $(M, g^{M})$ admits a local K\"ahler immersion into  $(\mathbb{C}^{\infty}, g_{0})$.
\end{lemma}

\subsection{The diastatic function of $(\Omega, g^{\Omega})$}

Let $\Omega$ be as in \eqref{equ:Omega}.
Let $g^{D}$ and $g^{\Omega}$ be K\"ahler metrics whose K\"ahler potentials are $-\log \varphi(z)$ and $-\log(\varphi(z)-||z_{0}||^{2})$
respectively. The following lemma shows a  relation  between their diastatic functions.

\begin{lemma}\label{lem:diastasis} Let $h$ be a positive number.
 If $-\log\varphi$ is the special K\"ahler potential that determined by the diastatic function of $g^{D}$, then
 $-h\log(\varphi(z)-||z_{0}||^{2})$  is the special K\"ahler potential that determined by the diastatic function  of $hg^{\Omega}$.
\end{lemma}
\begin{proof}
  Let $\Phi$ be the extension of $-\log \varphi$ on $U\times \overline{U}$, i.e. $\Phi(z,\overline{z})=-\log \varphi(z)$, then $$D(z,w)=\Phi(z,\overline{z})+\Phi(w,\overline{w})-\Phi(z,\overline{w})-\Phi(w,\overline{z}).$$
  Since $D(0,w)=-\log\varphi(w)$, we know
  $\Phi(0,0)-\Phi(0,\overline{w})-\Phi(w,0)=0.$
Let $\widetilde{\Phi}$ be the extension of $-h\log (\varphi(z)-||z_{0}||^{2})$ that satisfies
 $$\widetilde{\Phi}((z_{0},z),(\overline{z}_{0},\overline{z}))
 =-h\log (\varphi(z)-||z_{0}||^{2})
 =-h\log (e^{-\Phi(z,\overline{z})}-||z_{0}||^{2}),$$ then
  $\widetilde{\Phi}((z_{0},z),(\overline{w}_{0},\overline{w}))= -h\log (e^{-\Phi(z,\overline{w})}-z_{0}\overline{w}^{t}_{0})$.
  Hence,
  \begin{eqnarray*}
 &&D((0,0),(w_{0},w))\\
 &=&\widetilde{\Phi}((0,0),(0,0))+\widetilde{\Phi}((w_{0},w),
 (\overline{w}_{0},\overline{w}))-\widetilde{\Phi}((0,0),(\overline{w}_{0},\overline{w}))-\widetilde{\Phi}((w_{0},w),(0,0))\\
  &=&h\Phi(0,0)+\Phi((w_{0},w),(\overline{w}_{0},\overline{w}))-h\Phi(0,\overline{w})-h\Phi(w,0)\\
  &=&-h\log (\varphi(w)-||w_{0}||^{2}).
  \end{eqnarray*}
  We complete the proof.
\end{proof}

In the following sections, we will focus on the  case that
$\Omega$ contains the origin and  $\Omega$ is circular with the origin. It implies that
 $D$  is also  circular with the origin and $\varphi(e^{\sqrt{-1}\theta}z)=\varphi(z)$.
Conversely, if  $D$ is circular with the origin and $\varphi(e^{\sqrt{-1}\theta}z)=\varphi(z)$, then
$\Omega$ is circular with the origin.

\subsection{Immersion in complex Euclidean space}
Let $f: (M, g^{M})\rightarrow (\mathbb{C}^{N}, g_{0})$, $N\leq \infty$,  be a K\"ahler immersion,
then $\sqrt{h}f$, $h>0$, gives a K\"ahler immersion of $(M, hg^{M})$ into $(\mathbb{C}^{N}, g_{0})$.
There are no difference to prove one of them.

 \begin{theorem}\label{thm:CES}
Let $\Omega$ be  as in \eqref{equ:Omega}. Suppose that $\Omega$ is
a simply connected circular domain with center zero and
the function $-\log\varphi$ is the special K\"ahler potential determined by the diastatic function of  $g^{D}$ .
Then $(\Omega, g^{\Omega})$ admits a full K\"ahler immersion into $(\mathbb{C}^{\infty}, g_{0})$ if and only
if $(D, g^{D})$ admits a K\"ahler immersion into $(\mathbb{C}^{N}, g_{0})$,
$N\leq \infty$.
\end{theorem}
\begin{proof} Although $\Omega$ is a bounded pseudoconvex Hartogs domain with
$d_{0}$-dimensional fibers, it suffices to prove the case $d_{0}=1$.
In fact, suppose it is true for $d_{0}=1, \cdots, k$.
Define
\begin{eqnarray}
\Omega_{k}&=&\Big\{(z_{01},\cdots,z_{0k}, z)\in \mathbb{C}^{k}\times D : |z_{01}|^{2}+\cdots+|z_{0k}|^{2}<\varphi(z)\Big\}.
\end{eqnarray}
 The K\"ahler potential of $g^{\Omega_{k}}$ is $-\log (\varphi(z)-|z_{01}|^{2}-\cdots-|z_{0k}|^{2})$.
Let $(D, g^{D})$
 be a K\"ahler submanifold of $(\mathbb{C}^{N}, g_{0})$. By the assumption, $(\Omega_{k}, g^{\Omega_{k}})$
 is a full K\"ahler submanifold of $(\mathbb{C}^{\infty}, g_{0})$.
We now prove it is also true for $d_{0}=k+1$.
Note that $\Omega_{k+1}$ can be written by  the following equivalent form.
 \begin{eqnarray*}
\Omega_{k+1}=\Big\{(z_{0(k+1)},z_{01},\cdots,z_{0k}, z)\in \mathbb{C}\times \Omega_{k} : |z_{0(k+1)}|^{2}<\varphi(z)-|z_{01}|^{2}-\cdots-|z_{0k}|^{2}\Big\}.
\end{eqnarray*}
Hence, $\Omega_{k+1}$ is a bounded pseudoconvex domain over $\Omega_{k}$ with $1$-dimensional fibers.
By the assumption,  $(\Omega_{k+1}, g^{\Omega_{k+1}})$
is a full K\"ahler submanifold of $(\mathbb{C}^{\infty}, g_{0})$.
Conversely,
if
$(\Omega_{k+1}, g^{\Omega_{k+1}})$
is a full K\"ahler submanifold of $(\mathbb{C}^{\infty}, g_{0})$,
then $(\Omega_{k}, g^{\Omega_{k}})$
is a full K\"ahler submanifold of $(\mathbb{C}^{\infty}, g_{0})$.
Moreover,
$(D, g^{D})$
is a K\"ahler submanifold of $(\mathbb{C}^{N}, g_{0})$.
Thus we prove that the result is also true for
 any finite dimensional fibers by induction.

Now we assume that $d_{0}=1$. Let $o=(o_{0},o_{1})$ be the origin of $\mathbb{C}\times\mathbb{C}^{d}$, $\eta=(z_{01}, z)=(z_{01}, z_{1},\cdots,z_{d})$ be the coordinate of the point $q\in \Omega$.
Consider the domain $\Omega$ with K\"ahler metric $g^{\Omega}$ whose globally defined K\"ahler potential around
the origin $o$ is
 \begin{equation}\label{equ:Doa}
\displaystyle{D(o,q)=-\log \big(\varphi(z)-|z_{01}|^{2}\big).}
\end{equation}

Let $p_{1}\in D$ be the projection point of $p$,
then $z$ is the complex coordinate of $p_{1}$.
Let $D(o_{1},q_{1})=-\log\varphi$ be the diastatic function for $g^{D}$ around the origin.
By Lemma \ref{lem:diastasis}, we know
$D(o,q)$ is the diastatic function for $g^{\Omega}$ around the origin.
The power expansion
\begin{equation}\label{equ:power expansion a}
D(o,q)=\sum\limits_{j,k=0}^{\infty} a_{j,k}(\eta)^{m_{j}}
(\overline{\eta})^{m_{k}}.
\end{equation}

The property of the matrix of coefficients can be described as follows:
the elements of the row vector have the same multi-index $m_{j}$,
and $|m_{k}|$ increases with the increase of  column number;
the elements of the column vector have the same multi-index $m_{k}$, and $|m_{j}|$   increases with the increase of
row number; The matrix of coefficient can be given by the following  block matrix.
\begin{equation}\label{equ:matrix a}
(a_{j,k})=
\left(\begin{array}{cccc}
A_{0,0}&A_{0,1} & A_{0,0}& \dots   \\
A_{1,0}&A_{1,1}&A_{1,2}& \dots \\
A_{2,0}&A_{2,1}&A_{2,2}& \dots\\
\vdots & \vdots & \vdots & \ddots\\
\end{array}
\right),
\end{equation}
where the element of matrix $A_{s,t}(s,t\in \mathbb{N})$ satisfies that $|m_{j}|=s$ and $|m_{k}|=t$.

The following observations tell us that the coefficients which satisfy some conditions are zero.
This  can be proved as follows.
\begin{description}
  \item[(1)]
Take the transformation $\Psi(z_{01},z)=(e^{i\theta}z_{01},z)$,
  by the expression \eqref{equ:Doa} of $D(o,q)$, we know $D(o,q)$ is invariant, i.e.
$$\displaystyle{\sum_{j,k=0}^{\infty} a_{j,k}(\eta)^{m_{j}}(\overline{\eta})^{m_{k}}}
=\sum_{j,k=0}^{\infty} a_{j,k}e^{i\theta(m_{j,1}-m_{k,1})}(\eta)^{m_{j}}(\overline{\eta})^{m_{k}},$$
then the coefficient $a_{j,k}=0$ if $m_{j,1}\neq m_{k,1}$.

  \item[(2)] Take the transformation $\Psi(z_{01},z)=(z_{01},e^{i\theta}z)$, then $D(o,q)=D(o,\Psi(q))$, i.e.,
\begin{eqnarray*}
\displaystyle{\sum_{j,k=0}^{\infty} a_{j,k}(\eta)^{m_{j}}(\overline{\eta})^{m_{k}}}
=\sum_{j,k=0}^{\infty} a_{j,k}e^{i\theta(|m_{j,2}+\cdots+m_{j,d+1}|- |m_{k,2}+\cdots+m_{k,d+1}|)}(\eta)^{m_{j}}(\overline{\eta})^{m_{k}}.
\end{eqnarray*}

This implies $a_{j,k}=0$ if $|m_{j,2}+\cdots+m_{j,d+1}|\neq |m_{k,2}+\cdots+m_{k,d+1}|$.
\end{description}

One direct result is that the block matrix \eqref{equ:matrix a} is a block diagonal matrix.
In fact, every element of $A_{s,t}(s\neq t)$ satisfies  $|m_{j}|\neq|m_{k}|$.
This means at list one of the following inequations is true,
 $m_{j,1}\neq m_{k,1}$ or
$|m_{j,2}+\cdots+m_{j,d+1}|\neq |m_{k,2}+\cdots+m_{k,d+1}|$.
Thus $A_{s,t}=0$ for $s\neq t$. Hence, the matrix \eqref{equ:matrix a} can be written as
\begin{equation*}
\big(a_{jk}\big)=
\left(\begin{array}{cccc}
A_{0,0}&0& 0& \dots   \\
0&A_{1,1}&0& \dots \\
0&0&A_{2,2}& \dots\\
\vdots & \vdots & \vdots & \ddots\\
\end{array}
\right),
\end{equation*}
where $A_{0,0}=-\log \varphi(0)=D(o_{1}, o_{1})$ and
\begin{equation*}
  A_{i,i}=
\left(\begin{array}{ccccc}
A_{z_{01}(i)}(0)&0     &\cdots &0& \\
0&A_{z_{01}(i-1)}(0)   &\cdots &0& \\
0&0                 &\cdots &0& \\
0&0        &\cdots&A_{z_{01}(0)}(0)\\
\end{array}
\right),
\end{equation*}
and $A_{z_{01}(\sigma)}(0), (\sigma=0,1,\cdots,i)$ contains derivatives $\partial(z_{01} z )^{m_{j}},$
$~\partial(\overline{z_{01}}\overline{ z} )^{m_{k}}$ of order $2i$ with $|m_{j}|=|m_{k}|=i$ such that
$m_{j,1}=m_{k,1}=\sigma$.
This also implies  the positive definite of matrix $(b_{j,k})$ is determined
 by metrics  $A_{z_{01}(\sigma)}(0), \sigma=0,1,\cdots,i$, where $i=1,2,\cdots,\infty.$
For convenience, define multi-index $\alpha_{j}=(m_{j,2},\cdots,m_{j,n})$.
Then $m_{j}=(m_{j,1},\alpha_{j})$, $m_{k}=(m_{k,1},\alpha_{k})$ and
$$\sum_{j,k=0}^{\infty} a_{j,k}(\eta)^{m_{j}}(\overline{\eta})^{m_{k}}
=\sum_{j,k=0}^{\infty} a_{j,k}(z_{01})^{m_{j,1}}(z)^{\alpha_{j}}(\overline{z}_{01})^{m_{k,1}}(\overline{z})^{\alpha_{k}},$$
where
\begin{eqnarray}
a_{j,k}
&=&\frac{\partial^{|m_{j}|+|m_{k}|}D(o,q)}
{\partial(z_{01})^{m_{j,1}}\partial(z)^{\alpha_{j}}
\partial(\overline{z}_{01})^{m_{k,1}}\partial(\overline{z})^{\alpha_{k}}}\bigg|_{(z_{01},z)=0}.
\end{eqnarray}

Now, we will study when  $A_{z_{01}(\sigma)}(0), \sigma=0,1,\cdots,i$ are positive semidefinite.
\begin{description}
  \item[(i)]  $A_{z_{01}(i)}(0)$ is always positive.
\begin{equation}\label{equ:Ai}
a_{j,k}=\displaystyle{\frac{\partial^{2i}D(o,q)}{\partial z_{01}^{i}\partial\overline{z}_{01}^{i}}\bigg|_{\eta=0}
=\Gamma(i)\Gamma(i+1)(\varphi(z)-|z_{01}|^{2})^{-i}\bigg|_{\eta=0}>0}.
\end{equation}
  \item[(ii)] Consider the matrices  $A_{z_{01}(\sigma)}(0),$ $\sigma=1,2,\cdots,i-1$. By direct computation,
\begin{eqnarray*}
\frac{\partial^{\sigma}D(o,q)}{\partial z_{01}^{\sigma}}
&=&\Gamma(\sigma)(\varphi(z)-|z_{01}|^{2})^{-\sigma}\overline{z}_{01}^{\sigma}, \\
\frac{\partial^{2\sigma}D(o,q)}{\partial z_{01}^{\sigma}\partial\overline{z}_{01}^{\sigma}}
&=&\Gamma(\sigma)\sum_{k+j=\sigma} \binom{\sigma}{k}\frac{\partial^{k}(\varphi(z)-|z_{01}|^{2})^{-\sigma}}
{\partial\overline{z}_{01}^{k}}\frac{\partial^{j}\overline{z}_{01}^{\sigma}}{\partial\overline{z}_{01}^{j}}.
\end{eqnarray*}
Hence, we have
\begin{eqnarray*}
a_{j,k}&=&\frac{\partial^{|m_{j}|+|m_{k}|}D(o,q)}{\partial
\eta^{m_{j}}\partial\overline{\eta}^{m_{k}}}\bigg|_{\eta=0}\\
&=&\Gamma(\sigma)\Gamma(\sigma +1)\frac{\partial^{|\alpha_{j}|+|\alpha_{k}|}
(\varphi(z)-|z_{01}|^{2})^{-\sigma}}
{\partial(z)^{\alpha_{j}}\partial(\overline{z})^{\alpha_{k}}}
\bigg|_{\eta=0}\\
&=&\Gamma(\sigma)\Gamma(\sigma +1)\frac{\partial^{|\alpha_{j}|+|\alpha_{k}|}\varphi(z)^{-\sigma}}
{\partial(z)^{\alpha_{j}}\partial(\overline{z})^{\alpha_{k}}}\bigg|_{z=0}.\\
\end{eqnarray*}

Thus $\varphi(z)^{-\sigma}-1=e^{\sigma D(o_{1},p_{1})}-1.$
By Lemma \ref{lem:zedda2009}, if $(D, g^{D})$ is a K\"ahler submanifold  of $(\mathbb{C}^{N}, g_{0})$,
then $(D, \sigma g^{D})$ is a K\"ahler submanifold  of $(\mathbb{CP}^{\infty}, g_{FS})$
for any $\sigma>0$. By Calabi's criterion,
$\sigma g^{D}$ is $1$-resolvable at $o_{1}\in D$.
This implies the metrics  $A_{z_{01}(\sigma)}(0),$ $\sigma=1,2,\cdots,i-1$ are all positive semidefinite.
  \item[(iii)] Consider the matrix $A_{z_{01}(0)}(0)$. Because
               \begin{eqnarray}
              a_{j,k}=\frac{\partial^{|m_{j}|+|m_{k}|}D(o,q)}{\partial
              (z)^{\alpha_{j}}\partial(\overline{z})^{\alpha_{k}}}\bigg|_{\eta=0}
                =\frac{\partial^{|m_{j}|+|m_{k}|}
                -\log\varphi(z)}{\partial (z)^{\alpha_{j}}
                (\overline{z})^{\alpha_{k}}}\bigg|_{z=0}.
                \end{eqnarray}
                Notice that
             $-\log\varphi(z)=D(o_{1},p_{1}).$ By the same reason in (ii),
             we know the metric  $A_{z_{01}(0)}(0)$ is  positive semidefinite
             if $(D, g^{D})$ admits a local K\"ahler immersion into $(\mathbb{C}^{N}, g_{0})$ in the neighbourhood of $o_{1}$.

\end{description}

In summary, if $(D, g^{D})$ admits a K\"ahler immersion into
$(\mathbb{C}^{N}, g_{0})$, then
$g^{\Omega}$  is resolvable of rank $\infty$ at $o\in \Omega$.
By  Calabi's criterion,
 $(\Omega, g^{\Omega})$ admits a local full K\"ahler immersion into $(\mathbb{C}^{\infty}, g_{0})$.
Since the domain is simple connected, the immersion can be extended to a global one. Hence,
$(\Omega, g^{\Omega})$ is a full K\"ahler submanifold  of $\mathbb{C}^{\infty}$.

The converse of this theorem  can be easily obtained by studying the leading minors in terms of (iii).
\end{proof}

\begin{remark}
By this theorem  and Lemma \ref{lem:zedda2009},
$(\Omega, hg^{\Omega})$ also admits a K\"ahler immersion of $(\mathbb{CP}^{\infty}, g_{FS})$ for any $h>0$
if $(D, g^{D})$ admits a K\"ahler immersion into $(\mathbb{C}^{N}, g_{0})$.
\end{remark}

\subsection{Immersion in complex projective space}
When the ambient space is  $(\mathbb{CP}^{\infty}, g_{FS})$, there is a special phenomenon:
although we have know $(M, g^{M})$ is a K\"ahler submanifold of $(\mathbb{CP}^{\infty}, g_{FS})$,
 it is still hard to determine whether or not $(M, hg^{M})$ is a K\"ahler submanifold of $(\mathbb{CP}^{\infty}, g_{FS})$ for a positive number  $h$.
In compact case,  the condition implies that $(M, g^{M})$  is also a K\"ahler submanifold of  $(\mathbb{CP}^{N}, g_{FS})$, $N<\infty$.
 Thus $h$ must be a positive integer.  In noncompact case,
 Loi and Zedda shows that the irreducible bounded symmetric domain  $D$ equipped with Bergman metric admits
 an equivalent K\"ahler immersion into  $(\mathbb{CP}^{\infty}, g_{FS})$ if and only  if $h\gamma\in W(D)\setminus \{0\}$, where
$\gamma$ denotes the genus of $D$, $W(D)$ denotes the Wallach set \eqref{equ:wallach set}.
 Our following result shows that
 the bounded circular pseudoconvex  Hartogs domains preserve the similar property of the base.

 \begin{theorem}\label{thm:CPS}
Let $\Omega$ be  as in \eqref{equ:Omega} and  $h$ be a positive number. Suppose that $\Omega$ is
a simply connected circular domain with center zero and
the function $-\log\varphi$ is the special K\"ahler potential determined by the diastatic function of  $g^{D}$.
Then $(\Omega, hg^{\Omega})$ admits a full K\"ahler immersion into  $(\mathbb{CP}^{\infty}, g_{FS})$
if and only if $(D, (h+\sigma)g^{D})$  admits a K\"ahler immersion into
$(\mathbb{CP}^{N}, g_{FS})$, $N\leq \infty$, for all $\sigma\in \mathbb{N}$.
\end{theorem}
\begin{proof}
By the same reason in the proof of Theorem \ref{thm:CES},
it suffices to prove the case $d_{0}=1$.
 Let $o=(o_{0},o_{1})$ be the origin of $\mathbb{C}\times\mathbb{C}^{d}$, $\eta=(z_{01}, z)=(z_{01}, z_{1},\cdots,z_{d})$ be the coordinate of the point $q\in \Omega$.
Consider the domain $\Omega$ with K\"ahler metric $hg^{\Omega}$, then the globally defined K\"ahler potential function around
the origin $o$ is
 \begin{equation}\label{equ:Dob}
\displaystyle{D(o,q)=-h\log \big(\varphi(z)-|z_{01}|^{2}\big).}
\end{equation}
By Lemma \ref{lem:diastasis}, we know
$D(o,q)$ is the diastatic function for $hg^{\Omega}$ around the origin.
The power expansion
\begin{equation}\label{equ:power expansion b}
e^{D(o,q)}-1
=(\varphi(z)-|z_{01}|^{2})^{-h}-1
=\sum\limits_{j,k=0}^{\infty} b_{j,k}(\eta)^{m_{j}}
(\overline{\eta})^{m_{k}}.
\end{equation}

The property of the matrix of coefficients can be described as follows:
the elements of the row vector have the same multi-index $m_{j}$,
and $|m_{k}|$ increases with the increase of  column number;
the elements of the column vector have the same multi-index $m_{k}$, and $|m_{j}|$   increases with the increase of
row number; The matrix of coefficient can be given by the following  block matrix.
\begin{equation}\label{equ:matrix b}
(b_{j,k})=
\left(\begin{array}{cccc}
B_{0,0}&B_{0,1} & B_{0,0}& \dots   \\
B_{1,0}&B_{1,1}&B_{1,2}& \dots \\
B_{2,0}&B_{2,1}&B_{2,2}& \dots\\
\vdots & \vdots & \vdots & \ddots\\
\end{array}
\right),
\end{equation}
where the element of matrix $B_{s,t}(s,t\in \mathbb{N})$ satisfies that $|m_{j}|=s$ and $|m_{k}|=t$.

The matrix  $(b_{j,k})$ has the similar properties of the matrix $(a_{j,k})$ in \eqref{equ:matrix a}.
Hence, the matrix \eqref{equ:matrix b} can be written as follows.
\begin{equation*}
\big(b_{jk}\big)=
\left(\begin{array}{cccc}
B_{0,0}&0& 0& \dots   \\
0&B_{1,1}&0& \dots \\
0&0&B_{2,2}& \dots\\
\vdots & \vdots & \vdots & \ddots\\
\end{array}
\right),
\end{equation*}
where $B_{0,0}=\varphi(0)^{-h}-1=e^{hD(o_{1}, o_{1})}-1$ and
\begin{equation*}
B_{i,i}=
\left(\begin{array}{ccccc}
B_{z_{01}(i)}(0)&0     &\cdots &0& \\
0&B_{z_{01}(i-1)}(0)   &\cdots &0& \\
0&0                 &\cdots &0& \\
0&0        &\cdots&B_{z_{01}(0)}(0)\\
\end{array}
\right),
\end{equation*}
and $B_{z_{01}(\sigma)}(0), (\sigma=0,1,\cdots,i)$ contains derivatives $\partial(z_{01} z )^{m_{j}},$
$~\partial(\overline{z}_{01}\overline{ z} )^{m_{k}}$ of order $2i$ with $|m_{j}|=|m_{k}|=i$ such that
$m_{j,1}=m_{k,1}=\sigma$.
This also implies  the positive definite of matrix $(b_{j,k})$ is determined
 by metrics  $B_{z_{01}(\sigma)}(0), \sigma=0,1,\cdots,i$, where $i=1,2,\cdots,\infty.$
For convenience, define $\alpha_{j}=(m_{j,2},\cdots,m_{j,n})$.
Then $m_{j}=(m_{j,1},\alpha_{j})$, $m_{k}=(m_{k,1},\alpha_{k})$ and
$$\sum_{j,k=0}^{\infty} b_{j,k}(\eta)^{m_{j}}(\overline{\eta})^{m_{k}}
=\sum_{j,k=0}^{\infty} b_{j,k}(z_{01})^{m_{j,1}}(z)^{\alpha_{j}}(\overline{z}_{01})^{m_{k,1}}(\overline{z})^{\alpha_{k}}.$$
where
\begin{eqnarray}
b_{j,k}
&=&\frac{\partial^{|m_{j}|+|m_{k}|}(\varphi(z)-|z_{01}|^{2})^{-h}-1}
{\partial(z_{01})^{m_{j,1}}\partial(z)^{\alpha_{j}}
\partial(\overline{z}_{01})^{m_{k,1}}\partial(\overline{z})^{\alpha_{k}}}\bigg|_{(z_{01},z)=0}.
\end{eqnarray}

We now study  under what conditions, $B_{z_{01}(\sigma)}(0), \sigma=0,1,\cdots,i$ can be  positive semidefinite.
\begin{description}
  \item[(i)]  $B_{z_{01}(i)}(0)$ is always positive.
$$b_{j,k}=\displaystyle{\frac{\partial^{2i}(\varphi(z)-|z_{01}|^{2})^{-h}}{\partial z_{01}^{i}\partial\overline{z}_{01}^{i}}\bigg|_{\eta=0}
=\frac{\Gamma(h+i)\Gamma(i+1)}{\Gamma(h)}(\varphi(z)-|z_{01}|^{2})^{-(h+i)}}\bigg|_{\eta=0}>0.$$
  \item[(ii)] Consider the matrices  $B_{z_{01}(\sigma)}(0),$ $\sigma=1,2,\cdots,i-1$. By direct computation,
 we have
\begin{eqnarray*}
b_{j,k}&=&\frac{\partial^{|m_{j}|+|m_{k}|}(\varphi(z)-|z_{01}|^{2})^{-h}}{\partial
\eta^{m_{j}}\partial\overline{\eta}^{m_{k}}}\bigg|_{\eta=0}\\
&=&\frac{\Gamma(h+\sigma)\Gamma(\sigma +1)}{\Gamma(h)}\frac{\partial^{|\alpha_{j}|+|\alpha_{k}|}
(\varphi(z)-|z_{01}|^{2})^{-(h+\sigma)}}
{\partial(z)^{\alpha_{j}}\partial(\overline{z})^{\alpha_{k}}}
\bigg|_{(z_{01},z)=0}\\
&=&\frac{\Gamma(h+\sigma)\Gamma(\sigma +1)}{\Gamma(h)}\frac{\partial^{|\alpha_{j}|+|\alpha_{k}|}\varphi(z)^{-(h+\sigma)}}
{\partial(z)^{\alpha_{j}}\partial(\overline{z})^{\alpha_{k}}}\bigg|_{z=0}.\\
\end{eqnarray*}
Since $\varphi(z)^{-(h+\sigma)}-1=e^{(h+\sigma)D(o_{1},p_{1})}-1.$
If $(D, (h+\sigma)g^{D})$ is a K\"ahler submanifold  of $(\mathbb{CP}^{\infty}, g_{FS})$ for $\sigma\in \mathbb{N}$,  by Calabi's criterion,
the  coefficient  matrix of power series of the right side
is positive semidefinite. This implies the metrics  $B_{z_{01}(\sigma)}(0),$ $\sigma=1,2,\cdots,i-1$ are all positive semidefinite.
  \item[(iii)] Consider the matrix $B_{z_{01}(0)}(0)$. Because
               \begin{eqnarray}
              b_{j,k}=\frac{\partial^{|m_{j}|+|m_{k}|}(\varphi(z)-|z_{01}|^{2})^{-h}}{\partial
              (z)^{\alpha_{j}}\partial(\overline{z})^{\alpha_{k}}}\bigg|_{(z_{01},z)=0}
                =\frac{\partial^{|m_{j}|+|m_{k}|}
                \varphi(z)^{-h}}{\partial (z)^{m_{j}}
                (\overline{z})^{m_{k}}}\bigg|_{z=0}.
                \end{eqnarray}
                Notice that
             $\varphi(z)^{-h}-1=e^{hD(o_{1},p_{1})}-1.$
             By the same reason in (ii), we know that
          the metric  $B_{z_{01}(0)}(0)$ is  positive semidefinite
          if $(D, hg^{D})$ is a K\"ahler submanifold  of $(\mathbb{CP}^{\infty}, g_{FS})$.
\end{description}

Finally, if $(D, (h+\sigma)g^{D})$  admits  a K\"ahler immersion into $(\mathbb{CP}^{\infty}, g_{FS})$ for all $\sigma\in \mathbb{N}$,
then $(\Omega, hg^{\Omega})$ admits a local full K\"ahler immersion into $(\mathbb{CP}^{\infty}, g_{FS})$.
Since $\Omega$ is simple connected, the immersion can be extended to a global one.

\ms
The converse of this theorem  can be easily obtained by (ii) and (iii).
\end{proof}

\begin{remark}
By Theorem  \ref{thm:CPS}, we know $(\Omega, hg^{\Omega})$ is  a K\"ahler submanifold  of $(\mathbb{CP}^{\infty}, g_{FS})$,
then   $(\Omega, (h+\sigma)g^{\Omega})$, $\sigma\in \mathbb{N}$, are all K\"ahler submanifolds  of $(\mathbb{CP}^{\infty}, g_{FS})$.
Suppose that there exists a positive number $h_{0}$ such $h_{0}g^{D}$ is
 not a K\"ahler submanifold of $(\mathbb{CP}^{\infty}, g_{FS})$.
 By Theorem  \ref{thm:CPS},  $(\Omega, h_{0}g^{\Omega})$ is  not a K\"ahler submanifold  of $(\mathbb{CP}^{\infty}, g_{FS})$.
 By Lemma \ref{lem:zedda2009},    $(\Omega, h_{0}g^{\Omega})$ is  not a K\"ahler submanifold  of $(\mathbb{C}^{\infty}, g_{0})$.
 By Lemma \ref{lem:Di2012},  $(\Omega, hg^{\Omega})$ is
    not a K\"ahler submanifold  of $(\mathbb{CH}^{\infty}, g_{hyp})$ for any $h>0$.
\end{remark}

\subsection{Immersion in complex hyperbolic space}\label{sec:3.5}

In this section, we deal with K\"ahler immersions of $(\Omega, hg^{\Omega})$ into complex hyperbolic space
 $(\mathbb{CH}^{N}, g_{hyp})$, $N\leq\infty$.

\begin{theorem}\label{thm:CHP}
Let $\Omega$ be  as in \eqref{equ:Omega}. Suppose that $\Omega$ is
a simply connected circular domain with center zero and
the function $-\log\varphi$ is the special K\"ahler potential  determined by the diastatic function of  $g^{D}$.
Then $(\Omega, hg^{\Omega})$
 is a full K\"ahler submanifold of $(\mathbb{CH}^{\infty}, g_{hyp})$ if and only
if $(D, hg^{D})$ is a K\"ahler submanifold of $(\mathbb{CH}^{N}, g_{hyp})$, $N\leq\infty$, and $0<h\leq1$.
\end{theorem}
\begin{proof}
Let $d_{0}=1$ and $o=(o_{0},o_{1})$ be the origin of $\mathbb{C}\times\mathbb{C}^{d}$, $\eta=(z_{01}, z)=(z_{01}, z_{1},\cdots,z_{d})$ be the coordinate of the point $q\in \Omega$.
Consider the domain $\Omega$ with K\"ahler metric $hg^{\Omega}$, then the globally defined K\"ahler potential function around
the origin $o$ is
 \begin{equation}\label{equ:Doc}
\displaystyle{D(o,q)=-h\log \big(\varphi(z)-|z_{01}|^{2}\big).}
\end{equation}
By Lemma \ref{lem:diastasis}, we know
$D(o,q)$ is the diastatic function for $hg^{\Omega}$ around the origin.  In Bochner coordinate,
\begin{equation}\label{equ:power expansion c}
1-e^{-D(o,q)}
=1-(\varphi(z)-|z_{01}|^{2})^{h}
=\sum\limits_{j,k=0}^{\infty} c_{j,k}(\eta)^{m_{j}}
(\overline{\eta})^{m_{k}}.
\end{equation}

The property of the matrix of coefficients can be described as follows:
the elements of the row vector have the same multi-index $m_{j}$,
and $|m_{k}|$ increases with the increase of  column number;
the elements of the column vector have the same multi-index $m_{k}$, and $|m_{j}|$   increases with the increase of
row number; The matrix of coefficient can be given by the following  block matrix.
\begin{equation}\label{equ:matrix c}
(c_{j,k})=
\left(\begin{array}{cccc}
C_{0,0}&C_{0,1} & C_{0,0}& \dots   \\
C_{1,0}&C_{1,1}&C_{1,2}& \dots \\
C_{2,0}&C_{2,1}&C_{2,2}& \dots\\
\vdots & \vdots & \vdots & \ddots\\
\end{array}
\right),
\end{equation}
where the element of matrix $C_{s,t}(s,t\in \mathbb{N})$ satisfies that $|m_{j}|=s$ and $|m_{k}|=t$.
The matrix  $(c_{j,k})$ has the similar properties of the matrix $(a_{j,k})$ in \eqref{equ:matrix a}.
Hence, the matrix \eqref{equ:matrix c} can be written as follows.
\begin{equation*}
\big(c_{jk}\big)=
\left(\begin{array}{cccc}
C_{0,0}&0& 0& \dots   \\
0&C_{1,1}&0& \dots \\
0&0&C_{2,2}& \dots\\
\vdots & \vdots & \vdots & \ddots\\
\end{array}
\right),
\end{equation*}
where $c_{0,0}=1-\varphi(0)^{h}=1-e^{-D(o_{1},o_{1})}$ and
\begin{equation*}
C_{i,i}=
\left(\begin{array}{ccccc}
C_{z_{01}(i)}(0)&0     &\cdots &0& \\
0&C_{z_{01}(i-1)}(0)   &\cdots &0& \\
0&0                 &\cdots &0& \\
0&0        &\cdots&C_{z_{01}(0)}(0)\\
\end{array}
\right),
\end{equation*}
and $C_{z_{01}(\sigma)}(0), (\sigma=0,1,\cdots,i)$ contains derivatives $\partial(z_{01} z )^{m_{j}},$
$~\partial(\overline{z}_{01}\overline{ z} )^{m_{k}}$ of order $2i$ with $|m_{j}|=|m_{k}|=i$ such that
$m_{j,1}=m_{k,1}=\sigma$.
This also implies  the positive definite of matrix $(c_{j,k})$ is determined
 by matrices  $C_{z_{01}(\sigma)}(0), \sigma=0,1,\cdots,i$, where $i=1,2,\cdots,\infty.$
For convenience, define $\alpha_{j}=(m_{j,2},\cdots,m_{j,n})$.
Then $m_{j}=(m_{j,1},\alpha_{j})$, $m_{k}=(m_{k,1},\alpha_{k})$ and
$$\sum_{j,k=0}^{\infty} c_{j,k}(\eta)^{m_{j}}(\overline{\eta})^{m_{k}}
=\sum_{j,k=0}^{\infty} c_{j,k}(z_{01})^{m_{j,1}}(z)^{\alpha_{j}}(\overline{z}_{01})^{m_{k,1}}(\overline{z})^{\alpha_{k}},$$
where
\begin{eqnarray}
c_{j,k}
&=&\frac{\partial^{|m_{j}|+|m_{k}|}(1-(\varphi(z)-|z_{01}|^{2})^{h})}
{\partial(z_{01})^{m_{j,1}}\partial(z)^{\alpha_{j}}
\partial(\overline{z}_{01})^{m_{k,1}}\partial(\overline{z})^{\alpha_{k}}}\bigg|_{(z_{01},z)=0}.
\end{eqnarray}

In the following two cases, we will study
sufficient and necessary conditions for  $C_{z_{01}(\sigma)}(0)$ are positive semidefinite.

\ms
\textbf{(1) The case that $h$ is a integer.}
\begin{description}
  \item[(i)] Consider $C_{z_{01}(i)}(0)$.
\begin{equation}
c_{j,k}= \left\{
\begin{array}{l}
 \frac{(-1)^{i+1}\Gamma(h+1)\Gamma(i+1)}{\Gamma(h-i+1)}(\varphi(z)-|z_{01}|^{2})^{h-i}|_{\eta=0} \ \ \ \ \text{for} \ i\leq h;\\
  0 \ \ \ \ \ \ \ \ \ \ \ \ \ \ \ \ \ \ \ \ \ \ \ \  \ \ \  \ \  \ \  \ \ \ \  \ \  \ \ \ \ \ \ \ \ \ \ \ \ \ \ \ \ \ \ \ \ \ \ \ \ \ \ \ \ \ \ \ \ \ \ \text{for} \ i>h.\\
\end{array}
\right.
\end{equation}
So we can get $C_{z_{01}(i)}(0)\geq 0$ for all $1\leq i< \infty$ if and only if $h=1$.

\ms
In the following, we only need to consider the case that $h=1$.

  \item[(ii)] The matrices  $C_{z_{01}(\sigma)}(0)\equiv0$ for $\sigma=1,2,\cdots,i-1$ when $h=1$.
  \item[(iii)] Consider the matrix $C_{z_{01}(0)}(0)$ when $h=1$. Because
               \begin{eqnarray*}
              c_{j,k}&=&-\frac{\partial^{|m_{j}|+|m_{k}|}(\varphi(z)-|z_{01}|^{2})}{\partial
              (z)^{\alpha_{j}}\partial(\overline{z})^{\alpha_{k}}}\bigg|_{(z_{01},z)=0}
                =-\frac{\partial^{|m_{j}|+|m_{k}|}
                \varphi(z)}{\partial (z)^{\alpha_{j}}
                (\overline{z})^{\alpha_{k}}}\bigg|_{z=0}\\
                &=&-\frac{\partial^{|m_{j}|+|m_{k}|}
                (1-e^{-D(o_{1},p_{1})})}{\partial (z)^{\alpha_{j}}
                (\overline{z})^{\alpha_{k}}}\bigg|_{z=0}.
                \end{eqnarray*}
\end{description}
By  Calabi's criterion and the discussion above,
 $(\Omega, g^{\Omega})$ admits a local K\"ahler immersion into $(\mathbb{CH}^{\infty}, g_{hyp})$
if and only if $(D, g^{D})$ admits a local K\"ahler immersion into $(\mathbb{CH}^{N}, g_{hyp})$, $N\leq\infty$.
By Calabi's criterion, we know the K\"ahler immersion can be extended to a global one.

\ms
\textbf{(2) The case that $h$ is not a integer.}
\begin{description}
  \item[(i)] Consider $C_{z_{01}(i)}(0)$.
\begin{equation*}
c_{j,k}= \left\{
\begin{array}{l}
 \frac{(-1)^{i+1}\Gamma(h+1)\Gamma(i+1)}{\Gamma(h-i+1)}(\varphi(z)-|z_{01}|^{2})^{h-i}|_{\eta=0} \ \  \ \ \ \ \ \  \ \
  \ \ \ \  \ \ \ \text{for} \ i<h;\\ \\
  \frac{(-1)^{2i-[h]}\Gamma(h+1)\Gamma(i-h)\Gamma(i+1)}{\Gamma (h-[h])\Gamma([h]-h+2)}(\varphi(z)-|z_{01}|^{2})^{h-i}|_{\eta=0} \  \ \ \ \ \text{for} \ i>h.\\
\end{array}
\right.
\end{equation*}
If $0<h<1$, then $h<1\leq i$. Thus $c_{j,k}>0$ for any $i\in \mathbb{N}^{+}$.
If $1<h<2$, then  $c_{j,k}<0$ for $i\geq 2$.
If $h>2$, then  $c_{j,k}<0$ for $i=2$.
Hence,
$C_{z_{01}(i)}(0)$ can  be always  non negative if and only if $0<h<1$.

  \item[(ii)] Consider the matrices  $C_{z_{01}(\sigma)}(0)$, $\sigma=1,2,\cdots,i-1$, when $0<h<1$.

 \begin{equation*}
c_{j,k}=
  \frac{(-1)^{2\sigma}\Gamma(h+1)\Gamma(\sigma-h)\Gamma(\sigma+1)}{\Gamma (h-[h])\Gamma([h]-h+2)}\frac{\partial^{|\alpha_{j}|+|\alpha_{k}|}\varphi(z)^{h-\sigma}}
{\partial(z)^{\alpha_{j}}\partial(\overline{z})^{\alpha_{k}}}\bigg|_{z=0}.\\
\end{equation*}

Notice that $\varphi(z)^{h-\sigma}-1=e^{(\sigma-h)D(o_{1},p_{1})}-1.$
By Lemma \ref{lem:Di2012} and Lemma \ref{lem:zedda2009},
$(\sigma-h)g^{D}$ is always $1$-resolvable if $g^{D}$ is a K\"ahler submanifold of $(\mathbb{CH}^{N}, g_{hyp})$.

  \item[(iii)] Consider the matrix $C_{z_{01}(0)}(0)$ when $0<h<1$. Because
               \begin{eqnarray*}
              c_{j,k}&=&-\frac{\partial^{|m_{j}|+|m_{k}|}(\varphi(z)-|z_{01}|^{2})^{h}}{\partial
              (z)^{\alpha_{j}}\partial(\overline{z})^{\alpha_{k}}}\bigg|_{(z_{01},z)=0}
                =-\frac{\partial^{|m_{j}|+|m_{k}|}
                \varphi(z)^{h}}{\partial (z)^{\alpha_{j}}
                (\overline{z})^{\alpha_{k}}}\bigg|_{z=0}\\
                &=&-\frac{\partial^{|m_{j}|+|m_{k}|}
                (1-e^{-hD(o_{1},p_{1})})}{\partial (z)^{\alpha_{j}}
                (\overline{z})^{\alpha_{k}}}\bigg|_{z=0}.
                \end{eqnarray*}
\end{description}
By (i) and Calabi's criterion,
 $(\Omega, hg^{\Omega})$ admits a local K\"ahler immersion in $\mathbb{CH}^{\infty}$ if and only
  $(D, hg^{D})$ admits a local K\"ahler immersion in $\mathbb{CH}^{N}$, $N\leq \infty$ for $0<h<1$.
By Calabi's criterion again, we know the K\"ahler immersion can be extended to a global one.

\ms
The converse of this theorem  can be easily obtained by  the discussion above.
\end{proof}

\begin{remark}
By Lemma \ref{lem:Di2012},  if
$(D, g^{\Omega})$ is a  K\"ahler submanifolds  of $(\mathbb{CH}^{N}, g_{hyp})$, $N\leq \infty$, then
$(D, g^{\Omega})$ is a  K\"ahler submanifolds  of $(\mathbb{C}^{\infty}, g_{0})$.
Thus we can reduce the problem of studying the existence of K\"ahler immersions from
 $(\Omega, g^{\Omega})$ into complex space forms
to the case in  Theorem \ref{thm:CES}.
\end{remark}

By (i) in Theorem \ref{thm:CES}, Theorem \ref{thm:CPS} and  Theorem \ref{thm:CHP},
it implies that there does not exist a K\"ahler immersion of  $(\Omega, g^{\Omega})$  into any finite dimensional complex space forms.
For the reader's convenience, we summarize the results obtained so far by the following table.
\begin{table}[!hbp]
  \centering
  \begin{minipage}[t]{0.4\linewidth}
  \caption{Conclusion}
  \label{tab:1}
\begin{tabular}{|l|c|c|c|c|c|c|}
  \hline
$\exists$  & $\mathbb{C}^{n}$ &$\mathbb{C}^{\infty}$ &  $\mathbb{CP}^{n}$ &$\mathbb{CP}^{\infty}$ &   $\mathbb{CH}^{n}$ &$\mathbb{CH}^{\infty}$ \\
 \hline
A &  $\times$ & $\surd$      &$\times$     &$\surd$ & $\times$& $-$ \\
  \hline

B  &$\times$      & $\times$ & $\times$    &$\surd$   &$\times$&$\times$\\
  \hline
C & $\times$       &  $\surd$     & $\times$ & $\surd$ & $\times$& $\surd$\\
  \hline
\end{tabular}
 \end{minipage}
\end{table}

The notations
A, B, C denote the sufficient conditions in Theorem \ref{thm:CES}, Theorem \ref{thm:CPS} and  Theorem \ref{thm:CHP} respectively.
$\surd$ denote ``There exists a K\"ahler immersion'', $\times$ denote `` There does not exist a K\"ahler immersion''. $-$ denote ``Not sure''.
By Table \ref{tab:1}, we can obtain Theorem \ref{thm:1.2} immediately.

\ms
Note that the property of circular plays an important role in our proof. It makes
the Bochner coordinate of the diastasis very simple. If we remove this property,
then the problem will be too complicate to deal with.
A possible way is to consider the minimal domain which is a generalization of  circular domain.
We will explore it in the near future.

\bibliographystyle{amsplain}

\providecommand{\bysame}{\leavevmode\hbox to3em{\hrulefill}\thinspace}
\providecommand{\MR}{\relax\ifhmode\unskip\space\fi MR }
\providecommand{\MRhref}[2]{
\href{http://www.ams.org/mathscinet-getitem?mr=#1}{#2}}
\providecommand{\href}[2]{#2}

\noindent {Yihong Hao:\quad
              Institute of Mathematics, AMSS, and National Center for Mathematics and
              Interdisciplinary Sciences, Chinese Academy of Sciences, Beijing \rm{100190}, PR China.
              Email: haoyihong@126.com}\\

\noindent {An Wang:\quad
              School of Mathematical Sciences, Capital Normal University, Beijing \rm{100048}, PR China.
             Email: wangan@cnu.edu.cn}

\end{document}